\documentclass{article}
\usepackage{amsmath}
  \usepackage{paralist}
  \usepackage{graphics}
  \usepackage{epsfig}
\usepackage{graphicx}  \usepackage{epstopdf}
\usepackage[colorlinks=true]{hyperref}

\hypersetup{urlcolor=blue, citecolor=red}

\usepackage{amsmath,amssymb,color}
\usepackage{amsthm}
\usepackage{graphicx}
\usepackage{pgfplots}
\usepackage{pgfplotstable}
\usepackage{caption}

\newtheorem{theorem}{Theorem}[section]
\newtheorem{lemma}[theorem]{Lemma}
\newtheorem{corollary}[theorem]{Corollary}
\newtheorem{remark}[theorem]{Remark}

\newtheorem{assumption}{Assumption}

\newcommand{\dx}{\,\text{\rm{}d}x}
\newcommand{\ds}{\,\text{\rm{}d}s}

\newcommand{\be}{\begin{equation}}
\newcommand{\ee}{\end{equation}}

\def \R{\mathbb{R}}

\def \eps{\varepsilon}
\def \Uad{{U_{\text{ad}}}}
\def \Uadrho{{U_{\text{ad}, \rho}}}

\def \dx{\mathrm{d}x}

\newcommand{\meas}{\operatorname{meas}}

\renewcommand{\phi}{\varphi}

\begin{document}

%\subjclass{Primary: 49M20; Secondary: 49K20, 49N45.}

%\keywords{optimal control, Tikhonov regularization, bang-bang controls, sufficient second-order conditions, semilinear elliptic equations}

%\email{frank.poerner@mathematik.uni-wuerzburg.de}
%\email{daniel.wachsmuth@mathematik.uni-wuerzburg.de}

%\thanks{$^*$ Corresponding author: Daniel Wachsmuth}

\date{}

\title{\vspace{-3cm}Tikhonov regularization of optimal control problems governed by semi-linear partial differential equations\footnote{This work was supported by the German Research Foundation DFG under project grant \mbox{Wa 3626/1-1}.}
}

\maketitle

\centerline{\scshape Frank P\"orner}
\medskip
{\footnotesize
 \centerline{Institut f\"ur Mathematik}
   \centerline{Universit\"at W\"urzburg}
   \centerline{D-97974 W\"urzburg, Germany}
   \centerline{frank.poerner@mathematik.uni-wuerzburg.de}
}

\medskip\medskip

\centerline{\scshape Daniel Wachsmuth }
\medskip
{\footnotesize
 \centerline{Institut f\"ur Mathematik}
   \centerline{Universit\"at W\"urzburg}
   \centerline{D-97974 W\"urzburg, Germany}
   \centerline{daniel.wachsmuth@mathematik.uni-wuerzburg.de}
}

\bigskip\medskip

%\centerline{(Communicated by the associate editor name)}

\begin{abstract}
In this article, we consider the Tikhonov regularization of an optimal control problem of semilinear partial differential equations with  box constraints on the control. We derive a-priori regularization error estimates for the control under suitable conditions.
These conditions comprise second-order sufficient optimality conditions as well as regularity conditions on
the control, which consists of a source condition and a condition on the active sets.
In addition, we show that these conditions are necessary for convergence rates under certain conditions.
We also consider sparse optimal control problems and derive regularization error estimates for them.
Numerical experiments underline the theoretical findings.



\end{abstract}

\section{Introduction}
We consider the following optimal control problem
\begin{equation}\label{eq:main_problem}\tag{\textbf{P}}
 \begin{split}
    \text{Minimize} &\quad J(u) = \frac{1}{2}\|y_u - y_d\|_{L^2(\Omega)}^2  \\
    \text{such that} &\quad u_a \leq u \leq u_b \quad \text{a.e. in } \Omega,
\end{split}
\end{equation}
where $y_u$ is the solution of the Dirichlet problem
\begin{equation}\label{eq:semipde}
 \begin{aligned}
    Ay + f(y) &= u &&\text{ in } \Omega, \\
    y &= 0 &&\text{ on } \partial \Omega.
\end{aligned}
\end{equation}
Here, $\Omega \subseteq \R^n$, $n \le 3$, is a bounded Lipschitz domain.
The equation \eqref{eq:semipde} is a semilinear elliptic equation with the operator $A$ defined by
\[
(Ay)(x) = \sum\limits_{i,j=1}^n \partial_{x_j}[ a_{ij}(x) \partial_{x_i} y(x) ], \quad x\in \Omega.
\]
The standing assumptions on the data of the problem will be made precise below.

Since the cost function $J$ only implicitly depends on $u$ through the solution $y$ of the state
equation, the control problem is not coercive with respect to $u$ in suitable spaces.
Optimal controls of \eqref{eq:main_problem} may exhibit a bang-bang structure, where the control constraints are active on the whole domain, i.e., $\bar u(x) \in \{u_a, u_b\}$ almost everywhere.
In addition, due to the nonlinear constraint \eqref{eq:semipde} the resulting optimal control problem is non-convex.
This makes the analysis and numerical solution of this problem challenging.
To address this issue, we investigate the Tikhonov regularization of the problem
given by: Minimize
\[
    J_\alpha(u) := \frac{1}{2}\|y - y_d\|_{L^2(\Omega)}^2 + \frac{\alpha}{2} \|u\|_{L^2(\Omega)}^2
\]
subject to the semilinear equation and the control box constraints. Here, $\alpha>0$ is the Tikhonov
regularization parameter.
Here, we are interested in convergence of solutions or stationary points $u_\alpha$ of the regularized problems for $\alpha\searrow0$.
Under suitable conditions, we prove in Section \ref{sec:convrates} convergence rates of the type
\[
 \|u_\alpha -\bar u \|_{L^2(\Omega)} = \mathcal O(\alpha^{d/2}) \text{ for } \alpha \searrow0,
\]
see Theorem \ref{thm:rates_u}. This is the main result of the paper, and it is the first convergence rate result
for regularization of optimal control problems subject to {\em nonlinear} partial differential equations.
In addition, we also derive necessary conditions for convergence rates. As it turns out, a certain source condition
is necessary to obtain convergence rates, see Section \ref{sec:necessary}.

In the subsequent analysis, we will make use of the second-order conditions developed by Casas \cite{casas2012}.
They require positive definiteness of the second-derivative $J''$ of the reduced cost functional
with respect to solutions of linearized equations, see \eqref{eq:ssc} below.
A second ingredient is a condition on the optimal control and adjoint state of the
original problem.
This condition was used earlier for convex problems to prove convergence rates for Tikhonov regularization in \cite{wachsmuth2011}.
The present paper continues these investigations and generalizes the convergence rate results to optimal control problems
with non-linear state equations.

We also investigate sparse control problems given by
\[
 \begin{split}
    \text{Minimize} &\quad J(u) = \frac{1}{2}\|y_u - y_d\|_{L^2(\Omega)}^2  +  \beta \|u\|_{L^1(\Omega)}\\
    \text{such that} &\quad u_a \leq u \leq u_b \quad \text{a.e. in } \Omega,
\end{split}
\]
where $\beta>0$ is a parameter.
This is a non-smooth variant of the control problem above.
Again we study the Tikhonov regularization and derive error estimates,
see Section \ref{sec:sparsity}.

Optimal control of semi-linear partial differential equations has been intensively studied in the literature,
we refer to the monograph \cite{troeltzsch2010}.
In recent years, there is also a growing interest in sparse optimal control problems starting with \cite{stadler},
see also \cite{casas2012,CasasHerzogWachsmuth2012}.
Tikhonov regularization and its convergence was studied in \cite{daniels2016,daniels2017,wachsmuth2011c}
in connection with linear-quadratic optimal control problems.
As we show in this paper, the results obtained for linear equations can be carried over using similar techniques while
heavily relying on the second-order condition of Casas \cite{casas2012}.

The work on regularization of optimal control problems is certainly connected to
regularization of nonlinear inverse problems: If no control constraints are present, i.e., $\Uad=L^2(\Omega)$, the
problem \eqref{eq:main_problem} is an heat source identification problem, which amounts to a nonlinear, ill-posed
operator equation.
Tikhonov regularization of nonlinear equations is studied, e.g., in the monograph \cite{engl1996}.
Necessary conditions for convergence rates for non-linear problems can be found in \cite{neubauer1989}.
Regularization of variational inequalities was studied in \cite{fengshan1998}.
In some sense, our results generalize results from inverse problems theory: If no control constraints
are present, our regularity conditions reduce to well-known source conditions.

The paper is structured as follows. In Section \ref{sec:ass} we introduce the necessary tools needed later for the convergence analysis, e.g., the second order sufficient condition and our regularity assumption. A stability analysis of the Tikhonov regularization for $\alpha \to 0$ is done in Section \ref{sec:tikconv}. The associated convergence rates are established in Section \ref{sec:convrates}. The regularity assumption is also necessary for the convergence rates, which is shown in Section \ref{sec:necessary}. In Section \ref{sec:sparsity} we extend our analysis to a sparsity promoting objective functional and establish convergence rates under a suitable modified regularity assumption. Numerical results are provided in Section \ref{sec:num}.

\section{Assumptions and preliminary results}\label{sec:ass}
In the sequel,  we will make use of the following assumptions, see \cite{casas2012}.
To shorten our notation, will denote the partial derivatives $\frac{\partial}{\partial y} f$ and $\frac{\partial^2}{\partial y^2} f$
by $f'$ and $f''$, respectively.

\begin{description}
\item [(A1)] We assume that $f$ satisfies
$f(\cdot,0) \in L^{\bar p}(\Omega),$ with $\bar p \geq n/2$, and
\[
f'(x,y) \geq 0 \quad \forall y\in \R, \text{ for a.a.\@ } x \in \Omega.
\]
For all $M > 0$ there exists a constant $C_{f,M} > 0$ such that
\[
\left|f'(x,y) \right|+\left|f''(x,y)\right| \leq C_{f,M} \quad \text{for a.a.\@ } x\in \Omega \text{ and } |y| \leq M.
\]
For every $M>0$ and $\eps > 0$ there exists $\delta > 0$, depending on $M$ and $\eps$, such that
\[
\left| f''(x,y_2) -  f''(x,y_1) \right| \leq \eps
\]
holds for all $y_1$, $y_2$ satisfying $|y_1|, |y_2| \leq M$, $|y_2-y_1| \leq \delta$, and for a.a.\@ $x \in \Omega$.
\item [(A2)] The coefficients of the operator $A$ satisfy $a_{ij} \in C(\bar \Omega)$. There exists some $\lambda_A > 0$ such that
\[
\lambda_A |\zeta|^2 \leq \sum\limits_{i,j=1}^n a_{ij}(x) \zeta_i \zeta_j \quad \forall \zeta \in \R^n, \quad \text{for a.a.\@ } x\in \Omega.
\]
\item [(A3)] We assume $y_d \in L^{\bar p}(\Omega),$ with $\bar p \geq n/2$. Moreover, $u_a,u_b\in L^\infty(\Omega)$
with $u_a\le u_b$ a.e.\@ on $\Omega$.
\end{description}

Under these assumptions we can establish the following results.
Existence and uniqueness of solutions of the state equations are well-known, see, e.g.\@ \cite[Thm. 2.1]{casasreyestro08}.

\begin{theorem}\label{thm:pdeunique}
For every $u \in L^p(\Omega)$ with $p > n/2$, the state equation \eqref{eq:semipde} has a unique solution $y_u \in H_0^1(\Omega) \cap C(\bar \Omega)$, and there is $c>0$ such that
\[
 \|y_u\|_{H_0^1(\Omega)} + \|y_u\|_{C(\bar \Omega)} \le c \|u\|_{L^p(\Omega)}\quad \forall u \in L^p(\Omega).
 \]
Moreover, the control-to-state mapping $S: L^p(\Omega) \to H_0^1(\Omega) \cap C(\bar \Omega)$  is of class $C^2$
and globally Lipschitz continuous.
\end{theorem}

For convenience, let us introduce the space $Y:=H_0^1(\Omega) \cap C(\bar \Omega)$ endowed
with the norm
\[
 \|y\|_Y:= \|y\|_{H_0^1(\Omega)} + \|y\|_{C(\bar \Omega)}.
\]
Then Theorem \ref{thm:pdeunique} implies the existence of $M>0$ such that
\begin{equation}\label{eq123}
 \|y_u\|_Y \le M \quad \forall u\in U_{ad}.
\end{equation}

In addition, $S$ maps weakly converging sequences to strongly converging sequences:

\begin{lemma}\label{lem:statestrongconv}
Let $(u_k)$ be a sequence in $\Uad$ converging weakly in $L^2(\Omega)$ to $u$. Then, the associated sequence of states $(y_k)$ converges strongly in $Y$ to $y_u$.
\end{lemma}
\begin{proof}
This is \cite[Thm. 2.1]{casasreyestro08}.
\end{proof}

\subsection{Existence of solutions}
The existence of solutions of the optimal control problem can be proved by classical arguments.

\begin{theorem}\label{thm:optsolution}
Problem \eqref{eq:main_problem} has at least one solution $\bar u$ with an associated state $\bar y \in H_0^1(\Omega) \cap C(\bar \Omega)$.
\end{theorem}

The derivatives of the control-to-state map $S$ can be characterized by the following systems.
Let $u\in L^p(\Omega)$ be given with $y_u:=S(u)$. Then $z:=S'(u)v$ is the unique weak
solution of
\begin{equation}\label{eq:Dcontrolstate}
\begin{aligned}
    Az +  f'(y)z &= v &&\text{ in } \Omega, \\
    z &= 0 && \text{ on } \partial \Omega.
  \end{aligned}
\end{equation}
In addition, let us introduce the adjoint state $p_u$ associated to $u$ as the unique weak solution
of the adjoint equation
\begin{equation}\label{eq:adjoint}
\begin{aligned}
A^\ast p +  f'(y)p &= y_u - y_d &&\text{ in } \Omega, \\
    p &= 0 &&\text{ on } \partial \Omega.
\end{aligned}
\end{equation}
Using these expressions, the
derivatives of the cost functional $J$ are given by the following lemma.

\begin{lemma}\label{lemma:derivJ}
The functional $J: L^2(\Omega) \to \R$ is of class $C^2$, and the first and second derivative is given by
\begin{align*}
J'(u)(v) &= \int\limits_\Omega p v \; \dx,\\
J''(u)(v_1,v_2) &= \int\limits_\Omega \left(  1 - f''(x,y)p  \right) z_{v_1} z_{v_2} \; \dx,
\end{align*}
where we used the notation $y:=y_u$, $p:=p_u$, $z_{v_i}:= S'(u)v_i$.
\end{lemma}

Let us recall the first-order necessary optimality conditions.

\begin{theorem}
 Let $\bar u$ be a local solution of problem \eqref{eq:main_problem}. Then there is $\bar y:=S(\bar u)$ and $\bar p:=p_{\bar u}$
 such that the following system is satisfied:
 \begin{equation}\label{eq:fo1}
\begin{aligned}
    A\bar y + f( \bar y) &= \bar u &&\text{ in } \Omega, \\
    \bar y &= 0 &&\text{ on } \partial \Omega,
\end{aligned}
\end{equation}
\begin{equation}\label{eq:fo2}
\begin{aligned}
    A^\ast \bar p +  f'( \bar y)\bar p &= \bar y - y_d &&\text{ in } \Omega, \\
    \bar p &= 0 &&\text{ on } \partial \Omega,
\end{aligned}
\end{equation}
\begin{equation}\label{eq:fo3}
J'(\bar u)(u - \bar u) \geq 0 \quad \forall u \in \Uad.
\end{equation}
\end{theorem}

Let us close this section with the following stability result
regarding the solutions of the adjoint equations.

\begin{lemma}\label{lem:PDEstab}
Let $\bar u\in U_{ad}$ be given with associated state $\bar y$ and adjoint state $\bar p$.
Then there is a constant $c>0$ such that for all $u \in U_{ad}$
it holds
\[
\|\bar p - p_u\|_Y \le c \|\bar y - y_u\|_{L^2(\Omega)}.
\]
\end{lemma}
\begin{proof}
Let us denote $y:=y_u$ and $p:=p_u$. Then the difference $p-\bar p$
of the adjoint states satisfies
\[
A^\ast (p-\bar p) + f'(y)(p-\bar p) = y-\bar y + (f'(\bar y)-f'(y))\bar p.
\]
Due to the Lax-Milgram theorem and Stampacchia's estimates \cite[Th\'eor\`eme 4.2]{stampac65}, there is $c>0$ such that
\[
 \|p-\bar p\|_Y \le \|y-\bar y + (f'( \bar y)-f'( y))\bar p\|_{L^2(\Omega)}.
\]
Since $\bar p$ is the solution of a linear elliptic equation with
right-hand side in $L^2(\Omega)$, we know $\bar p\in L^\infty(\Omega)$.
Hence, we can estimate using the assumptions on $f$
\[
 \|(f'( \bar y)-f'( y))\bar p\|_{L^2(\Omega)} \le C_{f,M}\|\bar y-y\|_{L^2(\Omega)} \|\bar p\|_{L^\infty(\Omega)}.
\]
with $M$ given by \eqref{eq123}. And the claim is proven
\end{proof}

\subsection{Sufficient second-order optimality conditions}

To formulate the sufficient second order conditions we will need the following notation. Following Casas~\cite{casas2012},  we define
for $\tau>0$
the extended critical cone at $\bar u\in U_{ad}$ by
\[
C_{\bar u}^\tau = \left\{ v \in L^2(\Omega): \
\begin{aligned} v(x)   \geq 0 &\text{ if } \bar u(x) = u_a(x)\\
v(x)\leq 0 & \text{ if } \bar u(x) = u_b(x) \\
v(x)= 0 & \text{ if } |\bar p(x)| > \tau
\end{aligned}
\right\}.
\]
The second-order condition now reads as follows.
{
\renewcommand{\theassumption}{\textbf{SOSC}}
\begin{assumption}[Second order sufficient condition]\label{ass:SecondOrder}
Let $\bar u \in \Uad$ be given. Assume that there exists $\delta > 0$ and $\tau > 0$ such that
\begin{equation}\label{eq:ssc}
J''(\bar u) v^2 \geq \delta \|z_v\|_{L^2(\Omega)}^2 \quad \forall v \in C_{\bar u}^\tau,
\end{equation}
where we used the notation $z_v = S'(\bar u) v$.
\end{assumption}
}

This condition together with the first-order necessary conditions imply local optimality, see \cite{casas2012}.

\begin{theorem}\label{thm:ssc}
Let us assume that $\bar u$ is a feasible control for problem \eqref{eq:main_problem} satisfying the first order optimality conditions \eqref{eq:fo1}--\eqref{eq:fo3} and the second order condition \ref{ass:SecondOrder}. Then, there exists $\eps > 0$ such that
\[
J(\bar u) + \frac{\delta}{9}\|y_u - \bar y\|_{L^2(\Omega)}^2 \leq J(u) \quad \forall u \in B_{\eps}(\bar u) \cap \Uad.
\]
\end{theorem}

\subsection{Regularity conditions}

In order to derive regularization error estimates for the control we assume some regularity on $\bar u$. We say, that $\bar u$ satisfies the assumption \ref{ass:ActiveSet}, if the following holds.

{
\renewcommand{\theassumption}{\textbf{ASC}}
\begin{assumption}[Active Set Condition]\label{ass:ActiveSet}
Let $\bar u$ be a local solution of \eqref{eq:main_problem}. Assume that there exists a set $I \subseteq \Omega$, a function $w \in Y$, and positive constants $\kappa, c$ such that the following holds:

\begin{enumerate}
  \item (source condition) $I \supset \{ x \in \Omega: \; \bar p(x) = 0 \}$  and
  \[\chi_I \bar u = \chi_I P_\Uad (S'(\bar u)^\ast w),\]
  \item (structure of active set) $A := \Omega \setminus I$ and for all $\eps > 0$
  \[ \meas\left(\{ x\in A: \; 0 < |\bar p(x)| < \eps  \}\right) \leq c\, \eps^\kappa,\]
  \item (regularity of solution) $S'(\bar u)^*w\in L^\infty(\Omega)$.
\end{enumerate}
\end{assumption}
}

This assumption is a combination of a source condition and a regularity assumption on the active sets. Similar regularity assumptions were used in, e.g., \cite{wachsmuth2016,wachsmuth2013,wachsmuth2011,wachsmuth2011b}
for problems with affine-linear control-to-state mapping $S$.
Note that for the special case $A = \Omega$ the solution $\bar u$ is of
bang-bang structure.
Under this regularity assumption we can establish an improved first order necessary condition, see \cite{wachsmuth2016}.

\begin{theorem}\label{thm:improvedfirstorder}
Let $\bar u$ satisfy assumption \ref{ass:ActiveSet}, then there is $c>0$ such that it holds
\[
J'(\bar u)(u - \bar u)  \ge c \|u-\bar u\|_{L^1(A)}^{1+\frac{1}{\kappa}} \quad \forall u \in \Uad.
\]
\end{theorem}

\section{Convergence of the Tikhonov regularization}\label{sec:tikconv}

Let us introduce the  Tikhonov regularized optimal control problem associated to \eqref{eq:main_problem}.
Let $\alpha>0$ be given. Then the regularized problem reads
\begin{equation}\label{eq:tikhonov_problem}\tag{$\boldsymbol{P_\alpha}$}
 \begin{split}
    \text{Minimize} &\quad J_\alpha(u) := \frac{1}{2}\|y - y_d\|_{L^2(\Omega)}^2 + \frac{\alpha}{2} \|u\|_{L^2(\Omega)}^2  \\
    \text{such that} &\quad u_a \leq u \leq u_b \quad \text{a.e. in } \Omega,
\end{split}
\end{equation}
where $y_u$ denotes again the solution of the semi-linear partial differential equation \eqref{eq:semipde}.
Clearly, the regularized problem admits solutions.

At first, we want to show that weak limit points of global solutions $(u_\alpha)_\alpha$ for $\alpha\searrow0$
are again global solutions of \eqref{eq:main_problem}.
In addition, we show that every strict local solution of \eqref{eq:main_problem} can be obtained
as a limit of local solutions of \eqref{eq:tikhonov_problem}. The results and the proofs are very similar to \cite[Section 4]{CasasTroeltzsch2014}, but since the proofs are short we present them here.

\begin{lemma}\label{lem:strongglobalconv}
Let $(u_\alpha)_{\alpha > 0}$ be a family of global solutions of \eqref{eq:tikhonov_problem} such that $u_\alpha \rightharpoonup u_0$ in $L^2(\Omega)$. Then $u_0$ is a global solution of \eqref{eq:main_problem}. In addition,
$u_\alpha \to u_0$ strongly in $L^2(\Omega)$. Moreover, the following identity holds
\[
\|u_0\|_{L^2(\Omega)} = \min\big\{ \|u\|_{L^2(\Omega)}: \; \text{$u$ is a global solution of \eqref{eq:main_problem}} \big\}.
\]
\end{lemma}

\begin{proof}
Let $u \in \Uad$ be given. Then it holds $J_\alpha(u_\alpha) \leq J_\alpha(u)$ for all $\alpha > 0$.
The family $(u_\alpha)_\alpha$ is bounded in $L^\infty(\Omega)$. Then Lemma \ref{lem:statestrongconv} implies
\[
J_0(u_0) = \lim\limits_{\alpha \to 0} J_0(u_\alpha) = \lim\limits_{\alpha \to 0} J_\alpha(u_\alpha) \leq \lim\limits_{\alpha \to 0} J_\alpha(u) = J_0(u).
\]
Since $u \in \Uad$ was arbitrary, it follows that $u_0$ is a global solution of \eqref{eq:main_problem}. Let us now prove the strong convergence $u_\alpha \to u_0$ in $L^2(\Omega)$. On one hand, we have due to the weakly lower semicontinuity of the norm that
\[
\|u_0\|_{L^2(\Omega)} \leq \liminf\limits_{\alpha \to 0} \|u_\alpha\|_{L^2(\Omega)} \leq \limsup\limits_{\alpha \to 0} \|u_\alpha\|_{L^2(\Omega)}.
\]
On the other hand, using that $u_0$ is a global solution of \eqref{eq:main_problem}, we obtain
\[
J_0(u_\alpha) + \frac{\alpha}{2} \|u_\alpha\|_{L^2(\Omega)}^2 = J_\alpha(u_\alpha) \leq J_\alpha(u_0) = J_0(u_0) + \frac{\alpha}{2} \|u_0\|_{L^2(\Omega)}^2 \leq J_0(u_\alpha) + \frac{\alpha}{2} \|u_0\|_{L^2(\Omega)}^2
\]
which implies $\|u_\alpha\|_{L^2(\Omega)} \leq \|u_0\|_{L^2(\Omega)}$ for all $\alpha > 0$. This shows  $\|u_\alpha\|_{L^2(\Omega)} \to \|u_0\|_{L^2(\Omega)}$, and $u_\alpha \to u_0$ in $L^2(\Omega)$ follows.

Let now $u$ be a global solution of \eqref{eq:main_problem}. Then we get
\begin{align*}
J_0(u_\alpha) + \frac{\alpha}{2}\|u_\alpha\|_{L^2(\Omega)}^2 &= J_\alpha(u_\alpha)\\
&\leq J_\alpha(u) =J_0(u) + \frac{\alpha}{2}\|u\|_{L^2(\Omega)}^2\\
&= J_0(u_0) + \frac{\alpha}{2}\|u\|_{L^2(\Omega)}^2\\
&\leq J_0(u_\alpha) + \frac{\alpha}{2}\|u\|_{L^2(\Omega)}^2,
\end{align*}
which implies $\|u_\alpha\|_{L^2(\Omega)} \leq \|u\|_{L^2(\Omega)}$ for all $\alpha>0$. This shows
\[
\|u_0\|_{L^2(\Omega)} = \lim\limits_{\alpha \to 0} \|u_\alpha\|_{L^2(\Omega)} \leq \|u\|_{L^2(\Omega)},
\]
which finishes the proof.
\end{proof}

This result shows that weak limit points of global solutions of \eqref{eq:tikhonov_problem} are
global solutions of minimal norm of \eqref{eq:main_problem}.
Since this problem is non-convex in general, such minimal norm solutions may not be uniquely determined.

\begin{theorem}\label{thm:localstrongconv}
Let $\bar u$ be a strict local solution of \eqref{eq:main_problem}. Then there exist $\rho > 0$ and a  family $(u_\alpha)_{\alpha \in (0, \bar \alpha)}$ of local solutions of \eqref{eq:tikhonov_problem} such that $u_\alpha \to \bar u$ in $L^2(\Omega)$.
\end{theorem}

\begin{proof}
For $\rho>0$ define the auxiliary feasible set $\Uadrho := \Uad \cap \{ v \in L^2(\Omega): \; \|v - \bar u\|_{L^2(\Omega)} \leq \rho \}$.
Let $\rho > 0$ be such that $\bar u$ is the unique global minimum of $J_0$ in the set $\Uadrho$.
We investigate the following auxiliary problem:
\[
\min J_\alpha(u) \text{ subject to } u \in \Uadrho.
\]
For every $\alpha > 0$ let $u_{\rho, \alpha}$ be a global solution of this auxiliary problem.
By construction, the family $(u_{\rho, \alpha})$ is uniformly bounded in $L^\infty(\Omega)$. Hence we find a sequence $\alpha_k \to 0$ such that $u_{\rho, \alpha_k} \rightharpoonup u_0$ in $L^2(\Omega)$. Arguing as in the proof of Lemma \ref{lem:strongglobalconv},
it follows that $u_0$ is a global minimum of $J_0$ on $\Uadrho$ and $\|u_{\rho, \alpha} - u_0\|_{L^2(\Omega)} \to 0$.
Consequently, we obtain $u_0 = \bar u$, and it holds $\lim_{\alpha \to 0} u_{\rho, \alpha} = \bar u$ strongly in $L^2(\Omega)$.
This implies that there is $\bar \alpha$ such that $\|u_{\rho, \alpha} - \bar u\|_{L^2(\Omega)} < \rho$ for all $\alpha < \bar \alpha$. Thus, the controls $u_{\rho, \alpha}$ are local minima of $J_\alpha$ on $\Uad$ for all $\alpha \leq \bar \alpha$.
\end{proof}

Using the second-order optimality condition and the growth estimate of Theorem \ref{thm:ssc}, we
can establish the following a-priori error estimate for the states and adjoints.
Analogous results were obtained in \cite{wachsmuth2011} for the case of a linear state equation.

\begin{theorem}\label{thm:Apriorirates}
Let $\bar u$ be a local solution of \eqref{eq:main_problem} satisfying \ref{ass:SecondOrder}.
Let $(u_\alpha)_{\alpha \in (0, \bar \alpha)}$ be such that $u_\alpha\in \Uad$ and $u_\alpha \to \bar u$ in $L^2(\Omega)$
for $\alpha \searrow0$. Then it holds
\[
\|y_\alpha - \bar y\|_{L^2(\Omega)} = o(\sqrt{\alpha}), \quad \|p_\alpha - \bar p\|_{L^\infty(\Omega)} = o(\sqrt{\alpha}).
\]
\end{theorem}

\begin{proof}
Using Theorem \ref{thm:ssc} and the fact that $J_\alpha(u_\alpha) \leq J_\alpha(\bar u)$ we get
\begin{align*}
J_0(\bar u) + \frac{\delta}{9} \|y_\alpha - \bar y\|_{L^2(\Omega)}^2 + \frac{\alpha}{2}\|u_\alpha\|_{L^2(\Omega)}^2 & \leq J_0(u_\alpha) + \frac{\alpha}{2}\|u_\alpha\|_{L^2(\Omega)}^2 = J_\alpha(u_\alpha)\\
&\leq J_\alpha(\bar u) = J_0(\bar u) + \frac{\alpha}{2}\|\bar u\|_{L^2(\Omega)}^2.
\end{align*}
This implies
\[
\frac{\delta}{9}\|y_\alpha - \bar y\|_{L^2(\Omega)}^2 \leq \frac{\alpha}{2}\left( \|\bar u\|_{L^2(\Omega)}^2 - \|u_\alpha\|_{L^2(\Omega)}^2 \right).
\]
Using the strong convergence $u_\alpha \to \bar u$, we get
\[
\lim\limits_{\alpha \to 0} \frac{\|y_\alpha - \bar y\|_{L^2(\Omega)}}{\sqrt{\alpha}}  = \lim\limits_{\alpha \to 0} \frac{9}{2 \delta} \sqrt{ \|\bar u\|_{L^2(\Omega)}^2 - \|u_\alpha\|_{L^2(\Omega)}^2  } \to 0,
\]
which proves the first part of the claim. The second part follows directly from Lemma \ref{lem:PDEstab}.
\end{proof}

\section{Convergence rates}\label{sec:convrates}
The results of Theorems \ref{thm:localstrongconv} and \ref{thm:Apriorirates} provide convergence results and a-priori rates.
However, numerical computations reveal that the a-priori rates are suboptimal, see, e.g., the numerical examples in Section \ref{sec:num}.
In addition, it is hard to guarantee that optimization algorithms deliver globally or locally optimal controls.
Hence, we will assume in the subsequent analysis that only stationary points $u_\alpha$ of \eqref{eq:tikhonov_problem} are available.
Recall that $u_\alpha$ is a stationary point if
it satisfies
\[
J'(u_\alpha)(u-u_\alpha) + (\alpha u_\alpha, \ u-u_\alpha)\ge 0 \quad \forall u\in\Uad.
\]
Furthermore one observes that in many applications the optimal control $\bar u$ exhibits a bang-bang structure,
as $y_d$ is not reachable, i.e., there exists no feasible control $u \in \Uad$ such that $y_d = Su$.
In this section we want to prove convergence rates under our regularity assumption \ref{ass:ActiveSet}, which is suitable for bang-bang solutions.
The regularity assumption \ref{ass:ActiveSet} was used in \cite{wachsmuth2016,wachsmuth2013,wachsmuth2011,wachsmuth2011b} to establish convergence rates for an affine-linear control-to-state mapping.
First we need some technical results, which will be helpful later on.

\begin{lemma}\label{lem41}
Let $\bar y=S(\bar u)$, $\bar u\in \Uad$ be given. Then there is $c>0$ and $\epsilon>0$ such that
\[
 \|y_u - \bar y\|_{L^2(\Omega)} \le c  \|z_{u-\bar u}\|_{L^2(\Omega)}
\]
holds for all $y_u$ with $u\in \Uad$ and $\|y_u - \bar y\|_{L^2(\Omega)} \le \epsilon$.
\end{lemma}
\begin{proof}
This can be proven following the lines of \cite[Corollary 2.8]{casas2012}.
\end{proof}

The following Lemma is an extension of \cite[Lemma 2.7]{casas2012}.

\begin{lemma}\label{lemma}
Let $(u_\alpha)_\alpha$ be a family of controls $u_\alpha \in \Uad$ such that $u_\alpha \rightharpoonup \bar u$ in $L^2(\Omega)$ for $\alpha\searrow0$.
Then for every $\eps>0$ there is $\alpha_{\max}>0$ such that
\[
 |J''(u_\alpha)v^2 - J''(\bar u)v^2 | \le \eps \|z_v\|_{L^2(\Omega)}^2
\]
for all $\alpha \in (0,\alpha_{\max})$.
\end{lemma}
\begin{proof}
 Let us denote the states and adjoints corresponding to $u_\alpha$ and $\bar u$ by $y_\alpha$, $p_\alpha$, and $\bar y$, $\bar p$,
 respectively.
 Due to Lemma \ref{lem:statestrongconv} we obtain $y_\alpha \to \bar y$ and $p_\alpha\to\bar p$ in $L^\infty(\Omega)$.
 Let us define $z_{\alpha,v}:=S'(u_\alpha)v$ and $z_v:=S'(\bar u)v$. According to Lemma \ref{lemma:derivJ}
 we can write
 \[
  J''(u_\alpha)v^2 - J''(\bar u)v^2 = \int_\Omega (f''(\bar y)\bar p - f''(y_\alpha)p_\alpha) z_v^2 \dx
  +\int_\Omega (1- f''(y_\alpha)p_\alpha) (z_{\alpha,v}^2-z_v^2)\dx.
 \]
 Here, the absolute value of the first integral can be made smaller than $\eps/2 \|z_v\|_{L^2(\Omega)}^2$ for $\alpha$ small enough due to $y_\alpha \to \bar y$ and $p_\alpha\to\bar p$ in $L^\infty(\Omega)$.
 Let us observe that $(z_{\alpha,v})$ is uniformly bounded in $Y$.
 It remains to study the difference $z_{\alpha,v}-z_v$.
 This difference satisfies the differential equation
\[
 A(z_{\alpha,v}-z_v) + f'(y_\alpha)(z_{\alpha,v}-z_v) + (f'(y_\alpha)-f'(\bar y))z_v=0.
\]
Arguing as in Lemma \ref{lem:PDEstab} we find
\[
 \|z_{\alpha,v}-z_v\|_Y \le c \|f'(y_\alpha)-f'(\bar y)\|_{L^\infty(\Omega)} \|z_v\|_{L^2(\Omega)}.
\]
Note that the constant $c$ is independent of $y_\alpha$, which is a consequence of the non-negativity of $f'$.
This estimate also implies the existence of $c>0$ independent of $\alpha$ such that
\[
 \|z_{\alpha,v}\|_{L^2(\Omega)} \le c \|z_v\|_{L^2(\Omega)}.
\]
This shows that the integral
\[
 \left| \int_\Omega (1- f''(y_\alpha)p_\alpha) (z_{\alpha,v}+z_v)(z_{\alpha,v}-z_v)\dx \right|
\]
can be made smaller than $\eps/2 \|z_v\|_{L^2(\Omega)}^2$ for $\alpha$ small enough.
\end{proof}

The following result uses the regularity assumption on the optimal control.
\begin{lemma}\label{lem:techniquest}
Let $\bar u$ satisfy Assumption \ref{ass:ActiveSet}.
Then it holds for all $u\in U_{ad}$
\[
 (\bar u,\bar u -u)_{L^2(\Omega)}\le \|w\|_{L^2(\Omega)} \|z_{u-\bar u}\|_{L^2(\Omega)}
 + \|\bar u -S'(\bar u)^\ast w\|_{L^\infty(A)}  \|u-\bar u\|_{L^1(A)}.
 \]
\end{lemma}
\begin{proof}
We compute
\[\begin{split}
 (\bar u,\bar u -u)_{L^2(\Omega)}& = ( \bar u, (\bar u - u) \mid_I)_{L^2(\Omega)} + ( \bar u, (\bar u - u) \mid_A)_{L^2(\Omega)}\\
 &\le  (S'(\bar u)^\ast w,(\bar u -u)\mid_I)_{L^2(\Omega)} + (\bar u,(\bar u -u)\mid_A)_{L^2(\Omega)}\\
 &= (w, S'(\bar u)(\bar u -u))_{L^2(\Omega)} +  (\bar u -S'(\bar u)^\ast w ,(\bar u -u)\mid_A)_{L^2(\Omega)},
 \end{split}
\]
which yields the result.
\end{proof}

We now have everything at hand to establish convergence rates for the control. We want to point out, that we only need weak convergence of the sequence $(u_\alpha)_\alpha$.
\begin{theorem}\label{thm:rates_u}
Let $\bar u$ satisfy Assumption \ref{ass:ActiveSet}, and let the assumptions of Theorem \ref{thm:ssc} hold for $\bar u$.
Let $(u_\alpha)_\alpha$ be a family of stationary points converging weakly in $L^2(\Omega)$ to $\bar u$ for $\alpha\searrow0$.
Then it holds with $d:=\min(\kappa,1)$ for $\alpha\searrow 0$
\[\begin{aligned}
\|z_{u_\alpha-\bar u}\|_{L^2(\Omega)} &= \mathcal{O}(\alpha^{\frac{d+1}2}),\\
\|u_\alpha-\bar u\|_{L^1(A)} & = \mathcal{O}(\alpha^{\frac{\kappa (d+1)}{\kappa+1}}),\\
\|u_\alpha-\bar u\|_{L^2(\Omega)} &= \mathcal{O}(\alpha^{d/2}).
\end{aligned}\]
In the case $w=0$ or $A=\Omega$, these convergences rates are obtained with $d:=\kappa$.
\end{theorem}
\begin{proof}
By first-order optimality conditions of $u_\alpha$ we know
\begin{equation}\label{eq:variational_ua}
 J'(u_\alpha)(u-u_\alpha) + \alpha (u_\alpha, \, u-u_\alpha)_{L^2(\Omega)}\ge0 \quad \forall u\in U_{ad}.
\end{equation}
Due to the Assumption \ref{ass:ActiveSet}, Theorem \ref{thm:improvedfirstorder} gives
\[
 J'(\bar u)(u-\bar u) \ge c_A \|u-\bar u\|_{L^1(A)}^{1+\frac1\kappa}  \quad \forall u\in U_{ad}.
\]
Using $\bar u$ and $u_\alpha$ as test functions in these inequalities and adding them, yields
\[
 c_A \|u_\alpha-\bar u\|_{L^1(A)}^{1+\frac1\kappa}  + \alpha \|u_\alpha-\bar u\|_{L^2(\Omega)}^2
 \le \alpha (\bar u,\bar u -u_\alpha)_{L^2(\Omega)} + (J'(\bar u)-J'(u_\alpha))(u_\alpha-\bar u).
\]
Using Lemma \ref{lem:techniquest}, we obtain by Young's inequality
\[\begin{split}
 \alpha (\bar u,\bar u -u_\alpha)_{L^2(\Omega)} &\le \alpha \|w\|_{L^2(\Omega)} \|z_{u_\alpha-\bar u}\|_{L^2(\Omega)}
 + \alpha \|\bar u -S'(\bar u)^\ast w\|_{L^\infty(A)}  \|u_\alpha-\bar u\|_{L^1(A)}\\
 &\le  \alpha \|w\|_{L^2(\Omega)} \|z_{u_\alpha-\bar u}\|_{L^2(\Omega)} + \frac{c_A}2 \|u_\alpha-\bar u\|_{L^1(A)}^{1+\frac1\kappa} +C \alpha^{\kappa+1},
\end{split}\]
with $C>0$ independent of $\alpha$. By Taylor expansion,
we obtain
\[
 (J'(\bar u)-J'(u_\alpha))(u_\alpha-\bar u) = -J''(\bar u)(u_\alpha-\bar u)^2
 - \big(J''(\tilde u_\alpha) - J''(\bar u)\big)(u_\alpha-\bar u)^2,
\]
with $\tilde u_\alpha$ between $u_\alpha$ and $\bar u$.

Let us argue that $u_\alpha -\bar u$ is in the extended critical cone $C^\tau_{\bar u}$.
Since $u_\alpha \rightharpoonup \bar u$ in $L^2(\Omega)$, it follows
from Theorem \ref{thm:pdeunique}, Lemma \ref{lem:statestrongconv}, and Lemma \ref{lem:PDEstab}
that $p_\alpha \to \bar p$ in $L^\infty(\Omega)$.
Hence, we obtain
$|\alpha u_\alpha + p_\alpha|>\tau/2$
and $\operatorname{sign}(\alpha u_\alpha + p_\alpha)=\operatorname{sign}(\bar p)$
for all $\alpha$ sufficiently small on the set, where $|\bar p|>\tau$ is satisfied.
If we choose $\alpha$ small enough, then also $\tau/2 > \alpha \max(\|u_a\|_{L^\infty} , \|u_b\|_{L^\infty})$ holds.
The variational inequality \eqref{eq:variational_ua} implies
\[
u_\alpha = P_\Uad \left(- \frac{1}{\alpha} p_\alpha \right),
\]
which
yields $u_\alpha=\bar u$ on $|\bar p|>\tau$. Consequently, $u_\alpha - \bar u\in C^\tau_{\bar u}$ holds
for all $\alpha$ sufficiently small. Hence, we can apply the second-order condition \ref{ass:SecondOrder} on $\bar u$ to obtain
\[
 J''(\bar u)(u_\alpha-\bar u)^2 \ge \delta \|z_{u_\alpha - \bar u}\|_{L^2(\Omega)}^2.
\]
In addition (see Lemma \ref{lemma}), we find that
\[
 |J''(\tilde u_\alpha)v^2 - J''(\bar u)v^2 | \le \frac\delta{4}\|z_v\|_{L^2(\Omega)}^2.
\]
for all $\alpha$ sufficiently small. Collecting the estimates above, we get
\begin{align*}
c_A \|u_\alpha - \bar u\|_{L^1(A)}^{1 + \frac{1}{\kappa}} &+ \alpha \| u_\alpha - \bar u\|_{L^2(\Omega)}^2
\leq \alpha (\bar u,\bar u -u_\alpha)_{L^2(\Omega)} + ( J'(u_\alpha) - J'(\bar u) )(\bar u - u_\alpha)\\
&\leq \alpha \|w\|_{L^2(\Omega)}\|z_{u_\alpha - \bar u}\|_{L^2(\Omega)} + \frac{c_A}{2} \|u_\alpha - \bar u\|_{L^1(A)}^{1 + \frac{1}{\kappa}} + C \alpha^{\kappa +1}\\
&\quad - J''(\bar u)(u_\alpha - \bar u) - (J''(\tilde u_\alpha) - J''(\bar u))(u_\alpha - \bar u)^2\\
&\leq \frac{\alpha^2 \|w\|_{L^2(\Omega)}^2}{ \delta} - \frac{\delta}{2} \|z_{u_\alpha - \bar u}\|_{L^2(\Omega)}^2 + \frac{c_A}{2}\|u_\alpha - \bar u\|_{L^1(A)}^{1 + \frac{1}{\kappa}}.
\end{align*}
This yields
\begin{multline*}
\frac\delta{2} \|z_{u_\alpha-\bar u}\|_{L^2(\Omega)}^2
+ \frac{c_A}2 \|u_\alpha-\bar u\|_{L^1(A)}^{1+\frac1\kappa}
+ \alpha \|u_\alpha-\bar u\|_{L^2(\Omega)}^2
\le \delta^{-1}\|w\|_{L^2(\Omega)}^2 \alpha^2 +C \alpha^{\kappa+1},
\end{multline*}
which proves the claim.
\end{proof}

Convergence rates for the state and adjoint state can be now easily obtained.

\begin{corollary}\label{cor:rates_yp}
Let the assumptions of Theorem \ref{thm:rates_u} hold for $\bar u$. Denote $\bar y$ the associated state and $\bar p$ the adjoint state. Then it holds for $\alpha\searrow0$
\[
\|y_\alpha - \bar y\|_{L^2(\Omega)} = \mathcal{O}(\alpha^{\frac{d+1}2}), \quad
\|p_\alpha - \bar p\|_{L^\infty(\Omega)} = \mathcal{O}(\alpha^{\frac{d+1}2}),
\]
where $d$ is as in the statement of Theorem \ref{thm:rates_u}.
\end{corollary}
\begin{proof}
By Theorem \ref{thm:rates_u} we already know $\|z_{u_\alpha-\bar u}\|_{L^2(\Omega)} = \mathcal{O}(\alpha^{\frac{d+1}2})$.
Lemma \ref{lem41} implies
$\|y_\alpha - \bar y\|_{L^2(\Omega)} = \mathcal{O}(\alpha^{\frac{d+1}2})$ for $\alpha\searrow0$.
Lemma \ref{lem:PDEstab} then proves the claim for the convergence of the adjoint states.
\end{proof}

\begin{remark}
The convergence rates obtained in Theorem \ref{thm:rates_u} and Corollary \ref{cor:rates_yp} resemble the rates obtained for the control of a linear partial differential equation,
see \cite{daniels2016,daniels2017}, which improved on the results of \cite{wachsmuth2011}.
\end{remark}

\section{Necessity of the regularity condition}\label{sec:necessary}

In this section we will show that the regularity assumption \ref{ass:ActiveSet} is necessary to
obtain the convergence rates provided by Theorem \ref{thm:rates_u}.
In the case of a linear state equation,
such results were obtained in \cite{daniels2016,daniels2017,wachsmuth2013necessary}.
As it turns out, these results can be transferred to the nonlinear case with suitable modifications.

\begin{theorem}
Let us assume that $\{ x \in \Omega: \; \bar p(x) = 0 \} \subset A^c$ holds for some given set $A \subset \Omega$.
Furthermore assume that there exists a constant $\sigma > 0$ such that
\[
u_a(x) \leq - \sigma < 0 < \sigma \leq u_b(x) \quad \text{ f.a.a. } x \in \Omega.
\]
Let $(u_\alpha)_\alpha$ be a family of stationary points of \eqref{eq:tikhonov_problem}.
Suppose that
\[
 \|\bar u - u_\alpha\|_{L^1(A)} + \|\bar p - p_\alpha\|_{L^\infty(A)} = \mathcal{O}(\alpha^\kappa)
\]
for some $\kappa>1$ and all $\alpha > 0$ sufficiently small. Then there is $c>0$ such that the relation
\[
\meas\left(\{ x \in A: \; |\bar p(x)| \leq \eps \}\right) \leq c\, \eps^\kappa
\]
is fulfilled for all $\eps>0$ sufficiently small.
\end{theorem}
\begin{proof}
 The proof is analogous to that of the corresponding result \cite[Thm.\@ 13]{daniels2017}. As this proof only uses the variational inequality \eqref{eq:fo3},
 it can be transferred to our situation without modifications.
\end{proof}

Second, we will show that the source condition is satisfied on the inactive set $\{x \in \Omega: \; \bar p(x) = 0\}$
if the convergence rate is sufficiently large. For a related result concerning the regularization of
 an ill-posed nonlinear operator equation we refer to \cite{neubauer1989}.

\begin{theorem}
Let $(u_\alpha)_\alpha$ be a family of stationary points of \eqref{eq:tikhonov_problem}
converging weakly to $\bar u\in \Uad$ in $L^2(\Omega)$.
Suppose the convergence rate $\|y_\alpha - \bar y\|_{L^2(\Omega)} = \mathcal{O}(\alpha)$
holds for $\alpha\searrow0$.
Then there exists a function $w \in L^2(\Omega)$ such that $\bar u = P_\Uad\left( S'(\bar u)^\ast w \right)$ holds pointwise
almost everywhere on the set $K:=\{x \in \Omega: \; \bar p(x) = 0\}$.

If in addition $\|y_\alpha - \bar y\|_{L^2(\Omega)}= o(\alpha)$ holds, then $\bar u$ vanishes on $K$.
\end{theorem}
\begin{proof}
By assumptions, we know $u_\alpha \to \bar u$ in $L^2(\Omega)$, $\|y_\alpha - \bar y\|_{L^2(\Omega)} = \mathcal{O}(\alpha)$,
and $\|p_\alpha - \bar p\|_Y = \mathcal{O}(\alpha)$, which is a consequence of Lemma \ref{lem:PDEstab}.

Let $\hat u\in \Uad$ be given with $\hat u=\bar u$ on $\Omega\setminus K$. This implies $(\bar p,\,\hat u-\bar u)_{L^2(\Omega)}=0$
and
\[
 0\le (\bar p,\,u_\alpha-\bar u)_{L^2(\Omega)}= (\bar p,\,u_\alpha-\hat u)_{L^2(\Omega)}.
\]
Testing the variational inequality \eqref{eq:variational_ua} with $\tilde u$ and adding it to the above
leads to
\be\label{eq508}
 (\alpha u_\alpha + p_\alpha - \bar p, \hat u - u_\alpha)_{L^2(\Omega)} \geq 0.
\ee
As in the predecessor works mentioned above, the idea of the proof is to divide this inequality by $\alpha$ and then to pass to the limit $\alpha\searrow0$. Hence, we investigate the difference quotient $\frac1\alpha(p_\alpha-\bar p)$.
Using the defining equations of $p_\alpha$ and $\bar p$, we find that $p_\alpha-\bar p$ solves
\be\label{eq509}
\begin{aligned}
    A^\ast (p_\alpha-\bar p) +  f'(\bar y) ( p_\alpha -\bar p) + (f'(y_\alpha)-f'(\bar y))p_\alpha &=  y_\alpha -\bar y &&\text{ in } \Omega, \\
    p_\alpha -\bar p&= 0 &&\text{ on } \partial \Omega.
\end{aligned}
\ee
Let us write
\[
 f'(y_\alpha)-f'(\bar y) =  \int_0^1 f''(\bar y + s(y_\alpha-\bar y))\ds\ (y_\alpha-\bar y).
\]
Since $y_\alpha-\bar y$ is uniformly bounded in $L^\infty(\Omega)$ by Theorem \ref{thm:pdeunique}, the assumptions on $f$ and the Lebesgue dominated convergence theorem
imply $\int_0^1 f''(\bar y + s(y_\alpha-\bar y))\ds \to f''(\bar y)$ in $L^2(\Omega)$.

Let now $\dot y$ and $\dot p$ be subsequential weak limit points of $( \alpha^{-1}(y_\alpha-\bar y))$ and $( \alpha^{-1}(p_\alpha-\bar p))$
in $L^2(\Omega)$ and $H^1_0(\Omega)$, respectively. Dividing \eqref{eq509} by $\alpha$ and passing to the limit $\alpha\searrow0$
 yields
\[
\begin{aligned}
    A^\ast\dot p+  f'(\bar y)\dot p + f''(\bar y) \dot y \bar p  &=  \dot y&&\text{ in } \Omega, \\
    \dot p&= 0 &&\text{ on } \partial \Omega.
\end{aligned}
\]
Note that the assumptions imply $p_\alpha\to \bar p$ in $L^\infty(\Omega)$.
This shows
\[
 \dot p = S'(\bar u)^*(\dot y - f''(\bar y) \dot y \bar p) =:  S'(\bar u)^*w.
\]
with $w:=(1- f''(\bar y)\bar p) \dot y \in L^2(\Omega)$.
Dividing the variational inequality \eqref{eq508} by $\alpha$ and passing  to the limit $\alpha\searrow0$
we find
\[
 (\bar u + \dot p, \hat u - \bar u)_{L^2(\Omega)} \geq 0.
\]
Since $\hat u\in\Uad$ was arbitrary with the restriction $\hat u=\bar u$ on $\Omega\setminus K$,
this inequality implies
\[
 \chi_K \bar u  = \chi_K P_\Uad(-S'(\bar u)^*w).
\]
If in addition we have $\|y_\alpha - \bar y\|_{L^2(\Omega)} = o(\alpha)$, then we obtain $\|p_\alpha - \bar p\|_Y=o(\alpha)$. This implies that $\alpha^{-1}(p_\alpha-\bar p)$ converges to zero in $L^\infty(\Omega)$.
Passing to the limit in \eqref{eq508} gives $ \chi_K \bar u  = \chi_K P_\Uad(0)$, hence $\bar u=0$ holds almost everywhere on $K$.
\end{proof}

\begin{remark}
Let us point out an interesting reformulation of the source condition in terms of the Lagrangian.
To this end, let us introduce the Lagrange function to problem \eqref{eq:main_problem} by
\[
 \mathcal L(y,u,p):= J(y) - \langle Ay + f(y)-u,p\rangle.
\]
Then the result of the previous theorem can be written as: There exists $\dot y\in L^2(\Omega)$ such that
\[
  \chi_K \bar u  = \chi_K P_\Uad\Big(-S'(\bar u)^*(\mathcal L_{yy}(\bar u,\bar y,\bar p)\dot y)\Big).
\]
Here, $\mathcal L_{yy}$ denotes the partial derivative of second order of $L$ with respect to $y$
interpreted as a linear and continuous mapping from $L^2(\Omega)$ to $L^2(\Omega)$.

In case of a linear state equation, we obtain $L_{yy}=\mathrm{id}$. In this case, the theorem above reduces to the results obtained in \cite{wachsmuth2013necessary}.

In addition, the above results resemble results for nonlinear inverse problems from \cite{neubauer1989}.
Under the assumptions $\Uad=L^2(\Omega)$ and $\bar y= y_d$ (exact and attainable data), the source condition reduces to
\[
 \bar u = -S'(\bar u) \dot y.
\]
Here, we used that $\bar y=y_d$ implies $\bar p=0$ and $L_{yy}(\bar y,\bar u,\bar p)=\mathrm{id}$.
\end{remark}

\section{Extension to sparse control problems}\label{sec:sparsity}
In this section we consider the problem
\begin{equation}\label{eq:main_problem_sparse}\tag{\textbf{S}}
 \begin{split}
    \text{Minimize} &\quad F(u) = J(u) + \beta j(u) =  \frac{1}{2}\|y_u - y_d\|_{L^2(\Omega)}^2 + \beta \|u\|_{L^1(\Omega)}  \\
    \text{such that} &\quad u_a \leq u \leq u_b \quad \text{a.e. in } \Omega,
\end{split}
\end{equation}
with $J(u) = \frac{1}{2}\|y_u - y_d\|_{L^2(\Omega)}^2$, $j(u) =  \|u\|_{L^1(\Omega)}$, and $\beta>0$. The motivation for the additional $L^1$-term in the cost functional $F$ is the following. A solution $\bar u$ of \eqref{eq:main_problem_sparse} is sparse, i.e. large parts of $\bar u$ are identically zero. The larger $\beta$, the smaller the support of $\bar u$. One possible application of such a model is the optimal placement of controllers, since in many cases it is not desirable to control the system from the whole domain $\Omega$. Starting with the pioneering work \cite{stadler}, such sparsity related control problems have been studied in, e.g.,
\cite{wachsmuth2013necessary,wachsmuth2011c,wachsmuth2011b} for optimal control of linear partial differential equations and \cite{casas2012,CasasHerzogWachsmuth2012} for the optimal control of semi-linear equations.

In order to simplify the exposition, we assume $u_a(x) \leq 0 \leq u_b(x)$ almost everywhere in $\Omega$.
Our aim is to investigate so called bang-bang-off solutions, i.e., $\bar u \in \{u_a(x), 0, u_b(x)\}$ almost everywhere in $\Omega$.
The necessary optimality conditions for problem \eqref{eq:main_problem_sparse} are given by:
\begin{equation}\label{eq:fo1sp}
  \begin{aligned}
    A\bar y + f(\bar y) &= \bar u \text{ in } \Omega, \\
    \bar y &= 0 \text{ on } \partial \Omega,
\end{aligned}
\end{equation}
\begin{equation}\label{eq:fo2sp}
  \begin{aligned}
    A^\ast \bar p +  f'(\bar y)\bar p &= \bar y - y_d \text{ in } \Omega, \\
    \bar y &= 0 \text{ on } \partial \Omega,
\end{aligned}
\end{equation}
\begin{equation}\label{eq:fo3sp}
\int\limits_\Omega (\bar p + \beta \bar \lambda) (u - \bar u) \; \dx \geq 0 \quad \forall u \in \Uad
\end{equation}
with $\bar \lambda \in \partial \|\bar u\|_{L^1(\Omega)}$.
We refer to \cite{casas2012} for proofs.
Again we consider the Tikhonov regularization of problem \eqref{eq:main_problem_sparse}
given by
\begin{equation}\label{eq:tikhonov_problem_sparse}\tag{$\textbf{S}_\alpha$}
 \begin{split}
    \text{Minimize} &\quad F_\alpha(u) = \frac{1}{2}\|y_u - y_d\|_{L^2(\Omega)}^2 + \beta \|u\|_{L^1(\Omega)} + \frac{\alpha}{2}\|u\|_{L^2(\Omega)}^2 \\
    \text{such that} &\quad u_a \leq u \leq u_b \quad \text{a.e. in } \Omega.
\end{split}
\end{equation}

The following convergence result can be proven similarly to the related result of Theorem \ref{thm:localstrongconv}.
\begin{theorem}\label{thm:localstrongconv_sparse}
Let $\bar u$ be a strict local solution of \eqref{eq:main_problem_sparse}. Then there exist $\rho > 0$ and a  family $(u_\alpha)_{\alpha \in (0, \bar \alpha)}$ of local solutions of \eqref{eq:tikhonov_problem_sparse} such that $u_\alpha \to \bar u$ in $L^2(\Omega)$ and every $u_\alpha$ is a global minimum of $F_\alpha$ in $\Uadrho := \Uad \cap \{ v \in L^2(\Omega): \; \|v - \bar u\|_{L^2(\Omega)} \leq \rho \}$
\end{theorem}

Since $j$ is not twice differentiable, we follow \cite{casas2012} and consider the modified extended critical cone
defined by
\begin{multline*}
\tilde C_{u}^{\tau} = \Big\{ v \in L^2(\Omega): \; v(x)  \geq 0 \mbox{ if } u(x) = u_a(x),
\  v(x)\leq 0  \mbox{ if } u(x) = u_b(x),\\
 \text{and }  J'(u)v + \beta j'(u,v) \leq \tau \|z_v\|_{L^2(\Omega)} \Big\}.
\end{multline*}

The second order condition for the sparse control problems reads as follows:
{
\renewcommand{\theassumption}{\textbf{SSC 2}}
\begin{assumption}[Sufficient second order condition]\label{ass:SecondOrder_sparse}
Let $\bar u \in \Uad$ be given. Assume that there exists $\delta > 0$ and $\tau > 0$ such that
\[
J''(\bar u) v^2 \geq \delta \|z_v\|_{L^2(\Omega)}^2 \quad \forall v \in \tilde C_{u}^{\tau}.
\]
\end{assumption}
}

This second order condition induces local quadratic growth of the cost functional. The next theorem is due to \cite[Theorem 3.6]{casas2012}.

\begin{theorem}\label{thm:ssc_sparse}
Let us assume that $\bar u$ is a feasible control for problem \eqref{eq:main_problem_sparse} with state $\bar y$ and adjoint state $\bar p$ satisfying the first order optimality conditions \eqref{eq:fo1sp}--\eqref{eq:fo3sp} and the second order condition \ref{ass:SecondOrder_sparse}. Then, there exists $\eps > 0$ such that
\[
F(\bar u) + \frac{\delta}{5}\|z_{u - \bar u}\|_{L^2(\Omega)}^2 \leq F(u) \quad \forall u \in B_{\eps}(\bar u) \cap \Uad.
\]
\end{theorem}
The variational inequality \eqref{eq:fo3sp} implies
the following relations between $\bar u$ and $\bar p$
\[\bar u(x)  \begin{cases}
= u_a(x) & \text{if } \bar p(x) > \beta\\
\in [u_a(x),0] & \text{if } \bar p(x) = \beta\\
= 0 & \text{if } |\bar p(x)| < \beta\\
\in [0,u_b(x)] & \text{if } \bar p(x) = -\beta\\
= u_b(x) & \text{if } \bar p(x) <- \beta
\end{cases}\]
see \cite{casas2012,stadler}.
Hence,  we have to modify the regularity assumption \ref{ass:ActiveSet} to take the influence of the
non-smooth term $j$ into account, see also \cite{wachsmuth2013necessary,wachsmuth2011c}.

{
\renewcommand{\theassumption}{\textbf{ASC 2}}
\begin{assumption}[Active Set Condition]\label{ass:ActiveSet_sparse}
Let $\bar u$ be a local solution of \eqref{eq:main_problem} and assume that there exists a set $I \subseteq \Omega$, a function $w \in Y$, and positive constants $\kappa, c$ such that the following holds

\begin{enumerate}
  \item (source condition) $I \supset \{ x \in \Omega: \; |\bar p(x)| = \beta \}$  and
  \[\chi_I \bar u = \chi_I P_\Uad (S'(\bar u)^\ast w),\]
  \item (structure of active set) $A := \Omega \setminus I$ and for all $\eps > 0$
  \[|\{ x\in A: \; 0 < \big| |\bar p(x)| - \beta \big| < \eps  \}| \leq c \eps^\kappa,\]
  \item (regularity of solution) $S'(\bar u)^*w\in L^\infty(\Omega)$.
\end{enumerate}
\end{assumption}
}

%{
%\renewcommand{\theassumption}{\textbf{ASC 2}}
%\begin{assumption}[Active Set Condition]\label{ass:ActiveSet_sparse}
%Let $\bar u$ be a local solution of \eqref{eq:main_problem} and assume that there exists positive constants %$\kappa, c$ such that the following holds for all $\eps > 0$
%  \[|\{ x\in \Omega: \; 0 < | \bar p(x)| - \beta  < \eps  \}| \leq c \eps^\kappa.\]
%\end{assumption}
%}

Note that if $\bar u$ satisfies this condition with $A = \Omega$ it exhibits a bang-bang-off structure,  and the set $\{x \in \Omega: \; |\bar p(x)| = \beta\}$ is a set of measure zero. Again we can establish an improved first order necessary condition:

\begin{lemma}\label{lem:improvedfirstorder_sparse}
Let $\bar u$ satisfy assumption \ref{ass:ActiveSet_sparse}, then
\[
J'(\bar u)(u-\bar u) + \beta j'(\bar u; u - \bar u)   \ge c \|u-\bar u\|_{L^1(A)}^{1+\frac{1}{\kappa}} \quad \forall u \in \Uad.
\]
\end{lemma}

\begin{proof}
We start by using the directional derivative of the $L^1$-norm and compute
\[
J'(\bar u)(u-\bar u) + \beta j'(\bar u; u - \bar u) \geq \int\limits_{\{|\bar p| > \beta\}} (\bar p + \beta \bar \lambda)(u- \bar u) \; \dx + \int\limits_{\{|\bar p| < \beta\}} \bar p (u - \bar u) + \beta |u - \bar u| \; \dx.
\]
Let $\eps > 0$ be given. We now split the set $\{| \bar p| > \beta\}$ and derive
\begin{align*}
\int\limits_{\{|\bar p| > \beta\}}& (\bar p + \beta \bar \lambda)(u- \bar u) \; \dx \geq \int\limits_{\{|\bar p| > \beta + \eps\}} (\bar p + \beta \bar \lambda)(u- \bar u) \; \dx\\
&= \int\limits_{\{|\bar p| > \beta + \eps, \; \bar p < - \beta\}} (\bar p + \beta \bar \lambda)(u- \bar u) \; \dx + \int\limits_{\{|\bar p| > \beta + \eps, \; \bar p > \beta\}} (\bar p + \beta \bar \lambda)(u- \bar u) \; \dx\\
&= \int\limits_{\{-\bar p-\beta > \eps, \; \bar p < - \beta\}} \overbrace{|\bar p + \beta|}^{ \geq \eps} |u- \bar u| \; \dx + \int\limits_{\{\bar p-\beta > \eps, \; \bar p > \beta\}} \overbrace{|\bar p - \beta|}^{\geq \eps} |u- \bar u| \; \dx\\
&\geq \eps \int\limits_{\{-\bar p-\beta > \eps, \; \bar p < - \beta\}} |u- \bar u| \; \dx + \eps \int\limits_{\{\bar p-\beta > \eps, \; \bar p > \beta\}}  |u- \bar u| \; \dx\\
&= \eps \int\limits_{\{ |\bar p| \geq \beta + \eps \}} |u - \bar u| \; \dx.
\end{align*}
Similarly, we can estimate
\begin{align*}
\int\limits_{\{|\bar p| < \beta\}} &\bar p (u - \bar u) + \beta |u - \bar u| \; \dx \geq \int\limits_{\{|\bar p| < \beta - \eps\}} \bar p (u - \bar u) + \beta |u - \bar u| \; \dx\\
&= \int\limits_{\{|\bar p| < \beta-\eps, \; \bar p \geq 0\}} \bar p (u - \bar u) + \beta |u - \bar u| \; \dx + \int\limits_{\{|\bar p| < \beta-\eps, \; \bar p < 0\}} \bar p (u - \bar u) + \beta |u - \bar u| \; \dx\\
&\geq \int\limits_{\{\bar p < \beta-\eps, \; \bar p \geq 0\}} \overbrace{(-\bar p)}^{\geq \eps - \beta} |u - \bar u| + \beta |u - \bar u| \; \dx + \int\limits_{\{-\bar p < \beta-\eps, \; \bar p < 0\}} \overbrace{\bar p}^{\geq \eps - \beta} |u - \bar u| + \beta |u - \bar u| \; \dx\\
&\geq \int\limits_{\{\bar p < \beta-\eps, \; \bar p \geq 0\}} (\eps - \beta) |u - \bar u| + \beta |u - \bar u| \; \dx + \int\limits_{\{-\bar p < \beta-\eps, \; \bar p < 0\}} (\eps - \beta) |u - \bar u| + \beta |u - \bar u| \; \dx\\
&\geq \eps \int\limits_{\{\bar p < \beta-\eps, \; \bar p \geq 0\}}  |u - \bar u| \; \dx + \eps \int\limits_{\{-\bar p < \beta-\eps, \; \bar p < 0\}}   |u - \bar u| \; \dx\\
&= \eps \int\limits_{\{|\bar p| < \beta - \eps\}}  |u - \bar u| \; \dx.
\end{align*}
Let us define
\[
 A_\eps := \{x \in A: \; \big||\bar p(x)| -\beta\big| \ge  \eps\}.
\]
Then the above computations yield
\[
 J'(\bar u)(u-\bar u) + \beta j'(\bar u; u - \bar u) \ge \eps \|u- \bar u \|_{L^1(\eps)}.
\]
Let us note that assumption \ref{ass:ActiveSet_sparse} implies $|A \setminus A_\eps|\le c\, \eps^{\kappa }$.
Now, putting everything together,  we obtain using the regularity assumption on the active set
\begin{align*}
J'(\bar u)(u-\bar u) + \beta j'(\bar u; u - \bar u)
&\ge \eps \|u - \bar u\|_{L^1(A_\eps)}\\
&= \eps \|u - \bar u\|_{L^1(A)} -  \eps \|u - \bar u\|_{L^1(A \setminus A_\eps)}\\
&\geq \eps \|u - \bar u\|_{L^1(A)} -  \eps \|u - \bar u\|_{L^\infty(\Omega)} |A \setminus A_\eps|\\
&\geq \eps \|u - \bar u\|_{L^1(A)} -  c\, \eps^{\kappa + 1}.
\end{align*}
Here $c>1$ is a constant independent from $u$. Setting $\eps = c^{- \frac{2}{\kappa}} \|u - \bar u\|_{L^1(A)}^{\frac{1}{\kappa}}$
proves the claim.
\end{proof}

We  are now in the position to prove convergence rates. The proof mainly follows the proof of Theorem \ref{thm:rates_u}.
\begin{theorem}
Let $\bar u$ satisfy Assumption \ref{ass:ActiveSet_sparse} and let the assumptions of Theorem \ref{thm:ssc_sparse} hold for $\bar u$.
Let $(u_\alpha)_\alpha$ be a family of stationary points converging weakly in $L^2(\Omega)$ to $\bar u$.
Then it holds with $d=\min(\kappa,1)$ for $\alpha\searrow 0$ sufficiently small
\[\begin{aligned}
\|z_{u_\alpha-\bar u}\|_{L^2(\Omega)} &= \mathcal{O}(\alpha^{\frac{d+1}2}),\\
\|u_\alpha-\bar u\|_{L^1(A)} & = \mathcal{O}(\alpha^{\frac{\kappa (d+1)}{\kappa+1}}),\\
\|u_\alpha-\bar u\|_{L^2(\Omega)} &= \mathcal{O}(\alpha^{d/2}).
\end{aligned}\]
In the case $w=0$ or $A=\Omega$, these convergences rates are obtained with $d:=\kappa$.
\end{theorem}

\begin{proof}
We split the proof in two parts
and consider the two cases $u_\alpha - \bar u \in \tilde C_{\bar u}^\tau$ and $u_\alpha - \bar u \not\in \tilde C_{\bar u}^\tau$.

{\em (1) The case $u_\alpha - \bar u \in \tilde C_{\bar u}^\tau$.} The optimality conditions for $u_\alpha$ and $\bar u$ are given as
\begin{align}
(p_\alpha + \alpha u_\alpha, u - u_\alpha)_{L^2(\Omega)} + \beta j'(u_\alpha; u - u_\alpha) & \geq 0 \quad \forall u \in \Uad, \label{eq612}\\
(\bar p,  u - \bar u)_{L^2(\Omega)} + \beta j'(\bar u; u - \bar u) &\geq c_A \|u-\bar u\|_{L^1(A)}^{1 + \frac{1}{\kappa}} \quad \forall u \in\Uad
\label{eq613}.
\end{align}
Note that $j$ is a convex function, hence we have the identity
\[
j'(x; y-x) \leq j(y) - j(x),
\]
leading to
\[
j'(u_\alpha; \bar u - u_\alpha) + j'(\bar u; u_\alpha - \bar u) \leq 0.
\]
Testing \eqref{eq612} and \eqref{eq613} with $\bar u$ and $u_\alpha$, respectively,
we obtain
\begin{align*}
c_A \|u_\alpha - \bar u\|_{L^1(A)}^{1 + \frac{1}{\kappa}} + &\alpha \|u_\alpha - \bar u\|_{L^2(\Omega)}^2\\
&\leq \alpha (\bar u,\bar u -u_\alpha)_{L^2(\Omega)} + (p_\alpha - \bar p, \bar u - u_\alpha)_{L^2(\Omega)} \\
& \qquad + \beta ( j'(u_\alpha; \bar u - u_\alpha) + j'(\bar u; u_\alpha - \bar u) ) \\
&\leq \alpha (\bar u,\bar u -u_\alpha)_{L^2(\Omega)} + (J'(u_\alpha) -J'(\bar u))( \bar u - u_\alpha).
\end{align*}
As the regularity assumptions \ref{ass:ActiveSet} and \ref{ass:ActiveSet_sparse} only differ in
item (ii), Lemma \ref{lem:techniquest} is applicable here as well,
which gives with Young's inequality
\[\begin{split}
 \alpha (\bar u,\bar u -u_\alpha)_{L^2(\Omega)} &\le \alpha \|w\|_{L^2(\Omega)} \|z_{u_\alpha-\bar u}\|_{L^2(\Omega)}
 + \alpha \|\bar u -S'(\bar u)^\ast w\|_{L^\infty(A)}  \|u_\alpha-\bar u\|_{L^1(A)}\\
 &\le  \alpha \|w\|_{L^2(\Omega)} \|z_{u_\alpha-\bar u}\|_{L^2(\Omega)} + \frac{c_A}2 \|u_\alpha-\bar u\|_{L^1(A)}^{1+\frac1\kappa} +C \alpha^{\kappa+1},
\end{split}\]
with $C>0$ independent of $\alpha$. By Taylor expansion,
we obtain
\[
 (J'(\bar u)-J'(u_\alpha))(u_\alpha-\bar u) = -J''(\bar u)(u_\alpha-\bar u)^2
 - \big(J''(\tilde u_\alpha) - J''(\bar u)\big)(u_\alpha-\bar u)^2,
\]
with $\tilde u_\alpha$ between $u_\alpha$ and $\bar u$. Since $u_\alpha - \bar u \in \tilde C_{\bar u}^\tau$ we can apply the second-order condition on $\bar u$ to obtain
\[
 J''(\bar u)(u_\alpha-\bar u)^2 \ge  \delta\|z_{u_\alpha - \bar u}\|_{L^2(\Omega)}^2.
\]
By Lemma \ref{lemma}, we find that
\[
 |J''(\tilde u_\alpha)v^2 - J''(\bar u)v^2 | \le \frac\delta{4}\|z_v\|_{L^2(\Omega)}^2
\]
for all $\alpha$ sufficiently small. Altogether, we obtain
\begin{align*}
c_A \|u_\alpha - \bar u\|_{L^1(A)}^{1 + \frac{1}{\kappa}} &+ \alpha \| u_\alpha - \bar u\|_{L^2(\Omega)}^2 \leq \alpha (\bar u,\bar u -u_\alpha)_{L^2(\Omega)} + ( J'(u_\alpha) - J'(\bar u) )(\bar u - u_\alpha)\\
&\leq \alpha \|w\|_{L^2(\Omega)}\|z_{u_\alpha - \bar u}\|_{L^2(\Omega)} + \frac{c_A}{2} \|u_\alpha - \bar u\|_{L^1(A)}^{1 + \frac{1}{\kappa}} + C \alpha^{\kappa +1}\\
&\quad - J''(\bar u)(u_\alpha - \bar u)^2 - (J''(\tilde u_\alpha) - J''(\bar u))(u_\alpha - \bar u)^2\\
&\leq \alpha^2 \delta^{-1}\|w\|_{L^2(\Omega)}^2 - \frac{\delta}{2} \|z_{u_\alpha - \bar u}\|_{L^2(\Omega)}^2 + \frac{c_A}{2}\|u_\alpha - \bar u\|_{L^1(A)}^{1 + \frac{1}{\kappa}}
.
\end{align*}
This yields
\begin{multline*}
\frac\delta{2} \|z_{u_\alpha-\bar u}\|_{L^2(\Omega)}^2
+ \frac{c_A}2 \|u_\alpha-\bar u\|_{L^1(A)}^{1+\frac1\kappa}
+ \alpha \|u_\alpha-\bar u\|_{L^2(\Omega)}^2
\le
\delta^{-1}\|w\|_{L^2(\Omega)}^2 \alpha^2 +C \alpha^{\kappa+1},
\end{multline*}
which implies the existence of $C>0$ such that
\[
\|z_{u_\alpha - \bar u}\|_{L^2(\Omega)}^2 + \|u_\alpha - \bar u\|_{L^1(A)}^{1 + \frac{1}{\kappa}} + \alpha \|u_\alpha - \bar u\|_{L^2(\Omega)} \leq C ( \alpha^{\kappa + 1} + \alpha^2)
\]
holds for all $\alpha$ sufficiently small.

{\em (2) The case $u_\alpha - \bar u \not\in \tilde C_{\bar u}^\tau$.}
By definition of the extended critical cone, we know
\[
J'(\bar u)(u_\alpha - \bar u) + \beta j'(\bar u; u_\alpha - \bar u) > \tau \|z_{u_\alpha - \bar u}\|_{L^2(\Omega)}.
\]
Combining this with Lemma \ref{lem:improvedfirstorder_sparse} yields
\[
J'(\bar u)(u_\alpha - \bar u) + \beta j'(\bar u; u_\alpha - \bar u) > \frac{\tau}{2} \|z_{u_\alpha - \bar u}\|_{L^2(\Omega)} + \frac{c_A}{2}\|u_\alpha - \bar u\|_{L^1(A)}^{1 + \frac{1}{\kappa}}.
\]
Using the expansion
\[
 (J'(\bar u)-J'(u_\alpha))(u_\alpha-\bar u) = -J''(\tilde u_\alpha)(u_\alpha-\bar u)^2
 \]
with $\tilde u_\alpha$ between $\bar u$ and $u_\alpha$,
we get similarly as in the first part of the proof
\begin{multline*}
\frac{\tau}{2}\|z_{u_\alpha - \bar u}\|_{L^2(\Omega)} + \frac{c_A}{2}\|u_\alpha - \bar u\|_{L^1(A)}^{1 + \frac{1}{\kappa}} + \alpha \|u_\alpha - \bar u\|_{L^2(\Omega)}^2\\
\leq \alpha \|w\|_{L^2(\Omega)} \|z_{u_\alpha - \bar u}\|_{L^2(\Omega)} + \frac{c_A}{4}\|u_\alpha - \bar u\|_{L^1(A)}^{1 + \frac{1}{\kappa}} + C \alpha^{\kappa + 1}
- J''(\tilde u_\alpha)(u_\alpha - \bar u)^2.
\end{multline*}
Using the structure of $J''$ and Lemma \ref{lemma}, we obtain
\[
\begin{split}
|J''(\tilde u)(u_\alpha - \bar u)^2|  &=
|J''(\bar u)(u_\alpha - \bar u)^2|  + |(J''(\tilde u)-J''(\bar u)(u_\alpha - \bar u)^2|\\
& \le c \|z_{u_\alpha - \bar u}\|_{L^2(\Omega)}^2
\end{split}
\]
for all $\alpha$ sufficiently small.
Hence, it holds
\begin{multline*}
\frac{\tau}{2}\|z_{u_\alpha - \bar u}\|_{L^2(\Omega)} + \frac{c_A}{4}\|u_\alpha - \bar u\|_{L^1(A)}^{1 + \frac{1}{\kappa}} + \alpha \|u_\alpha - \bar u\|_{L^2(\Omega)}^2\\
\leq \alpha^2 \|w\|_{L^2(\Omega)}^2 + C \alpha^{\kappa + 1}
+ c \|z_{u_\alpha - \bar u}\|_{L^2(\Omega)}^2.
\end{multline*}
Since $z_{u_\alpha - \bar u}\to 0$ in $L^2(\Omega)$, the following inequality is
satisfied for all $\alpha$ small enough
\[
\frac{\tau}{4}\|z_{u_\alpha - \bar u}\|_{L^2(\Omega)} + \frac{c_A}{4}\|u_\alpha - \bar u\|_{L^1(A)}^{1 + \frac{1}{\kappa}} + \alpha \|u_\alpha - \bar u\|_{L^2(\Omega)}^2\\
\leq \alpha^2 \|w\|_{L^2(\Omega)}^2 + C \alpha^{\kappa + 1},
\]
which implies the claim for the second case.
\end{proof}

\section{Numerical examples}\label{sec:num}
In this section we present numerical examples to support our theoretical results. We construct a bang-bang solution for the following optimal control problem:
\begin{equation}\label{eq:model_prob}
    \text{Minimize} \quad  \frac{1}{2}\|y - y_d\|_Y^2
\end{equation}
subject to
\[
 \begin{aligned}
 -\Delta y + f(y)& = u + e_\Omega \quad \text{in } \Omega\\
    				  y &= 0 \quad \text{on } \partial \Omega
    				  \end{aligned}
\]
and
\[
-1 \leq u \leq 1 \quad \text{a.e. in } \Omega.
\]
with $\Omega = (0,1)$. To solve the regularized optimal control problem numerically,
we use dolfin-adjoint \cite{dolfin_adjoint,dolfin_adjoint_opt} with linear finite elements on an equidistant mesh with $10^6$ cells. We make use of the adjoint equation
\[
 -\Delta \bar p + f'(\bar y)\bar p = \bar y - y_d
\]
and set:
\begin{align*}
\bar p(x) &:= \sin(2 \pi x)\\
\bar u(x) &:= -\text{sgn}(\bar p(x))\\
\bar y(x) &:= \sin(\pi x) \\
e_\Omega(x) &:= -\bar u(x) - \Delta \bar y(x) + f( \bar y(x)) \\
y_d(x) &:= \bar y(x) + \Delta \bar p(x) - f'( \bar y(x))\bar p(x)
\end{align*}
It is easy to check that $(\bar u, \bar y, \bar p)$ is a solution of \eqref{eq:model_prob}. Moreover, Assumption \ref{ass:ActiveSet} is satisfied with $A = \Omega$ and $\kappa = 1$, see \cite{hinze2012}. We expect to obtain the following convergence rate with respect to the $L^2$ norm:
\[
\|u_\alpha - u^\dagger\|_{L^2(\Omega)} = \mathcal{O}\left( \alpha^{\frac{1}{2}} \right).
\]
We test with 3 different nonlinearities
\begin{align*}
f_1(x,y) &:= \sin(y),\\
f_2(x,y) &:= -y + y^3,\\
f_3(x,y) &:= \exp(y).
\end{align*}

The results can be seen in Figure \ref{fig:ex1}, \ref{fig:ex2} and \ref{fig:ex3},
where we plotted the error $\|u_\alpha-\bar u\|_{L^2(\Omega)}$ for solutions $u_\alpha$ of the discretized and
regularized problem.
As expected, the theoretical convergence order is very well resolved.

% workaround to prevent figure from floating
\begin{center}
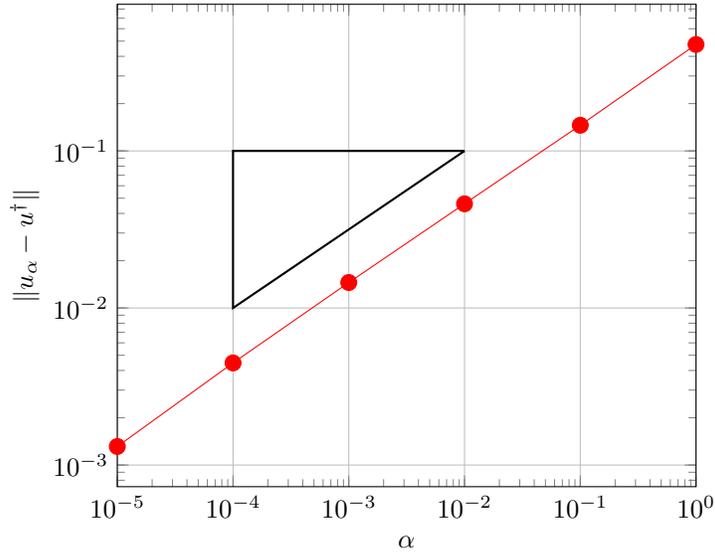

\begin{tikzpicture}
\begin{loglogaxis}[
    title=Example 1,
    xlabel={$\alpha$},
    ylabel style={align=center}, 
    ylabel=$\|u_\alpha - u^\dagger\|$,
    xminorticks=true,
    xmin=1e-5,
    xmax=1,
    %log basis x={2},
    height=8cm,
    grid=major,
    legend pos=north west,
]

\addplot[color=red, mark=*, mark options={scale=1.5}] plot coordinates {

(1e-0, 0.476142397202  )
(1e-1, 0.145704873952 )
(1e-2, 0.0460515358838  )
(1e-3, 0.0145221790711  )
(1e-4, 0.00446853328166  )
(1e-5, 0.00131250403435 )

};

\addplot[thick, color=black, no markers] plot coordinates {

(1e-2, 1e-1)
(1e-4, 1e-1)
(1e-4, 1e-2)
(1e-2, 1e-1)
};

\end{loglogaxis}
\end{tikzpicture}

\captionof{figure}{Error $\|u_\alpha - u^\dagger\|_{L^2(\Omega)}$ for $f_1$ in a double logarithmic plot for different values for $\alpha$. For comparison we plotted a triangle with slope $\frac{1}{2}$.}\label{fig:ex1}
\end{center}

%\begin{figure}[htbp]
%\makebox[\textwidth]{\input{tikz/ex1_u.tikz}}
%\caption{Error $\|u_\alpha - u^\dagger\|_{L^2(\Omega)}$ for $f_1$ in a double logarithmic plot for different %values for $\alpha$. For comparison we plotted a triangle with slope $\frac{1}{2}$.}
%\label{fig:ex1}
%\end{figure}

\begin{figure}[htbp]
\makebox[\textwidth]{\begin{tikzpicture}
\begin{loglogaxis}[
    title=Example 2,
    xlabel={$\alpha$},
    ylabel style={align=center}, 
    ylabel=$\|u_\alpha - u^\dagger\|$,
    xminorticks=true,
    xmin=1e-5,
    xmax=1,
    %log basis x={2},
    height=8cm,
    grid=major,
    legend pos=north west,
]

\addplot[color=red, mark=*, mark options={scale=1.5}] plot coordinates {

(1e-0,  0.476147224795 )
(1e-1,  0.145710684875)
(1e-2,  0.0460537097376 )
(1e-3,  0.0145217253411 )
(1e-4,  0.00446832736146 )
(1e-5, 0.00131056655159 )

};

\addplot[thick, color=black, no markers] plot coordinates {

(1e-2, 1e-1)
(1e-4, 1e-1)
(1e-4, 1e-2)
(1e-2, 1e-1)
};

\end{loglogaxis}
\end{tikzpicture}}
\caption{Error $\|u_\alpha - u^\dagger\|_{L^2(\Omega)}$ for $f_2$ in a double logarithmic plot for different values for $\alpha$. For comparison we plotted a triangle with slope $\frac{1}{2}$.}
\label{fig:ex2}
\end{figure}
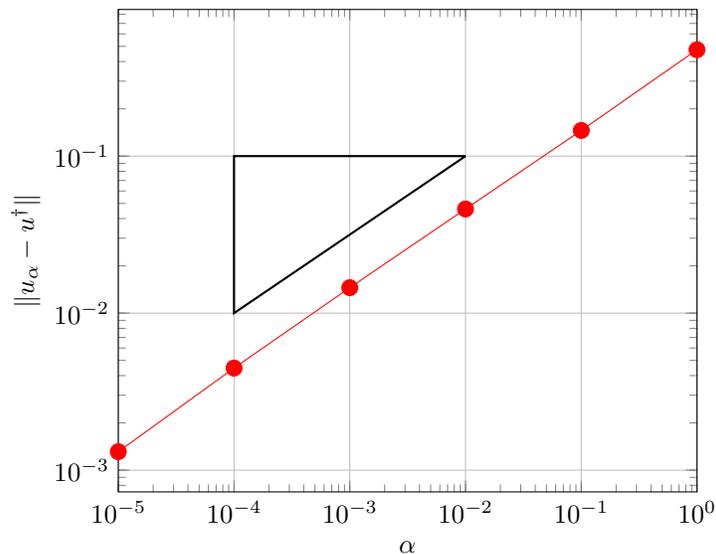

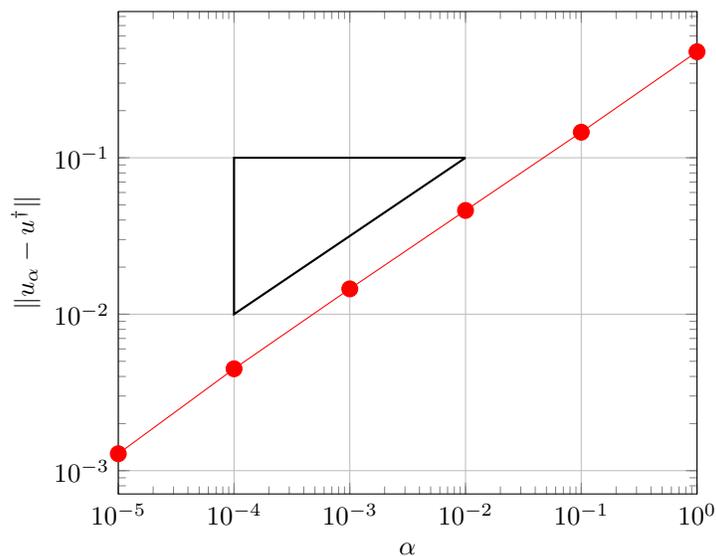
\begin{figure}[htbp]
\makebox[\textwidth]{\begin{tikzpicture}
\begin{loglogaxis}[
    title=Example 3,
    xlabel={$\alpha$},
    ylabel style={align=center}, 
    ylabel=$\|u_\alpha - u^\dagger\|$,
    xminorticks=true,
    xmin=1e-5,
    xmax=1,
    %log basis x={2},
    height=8cm,
    grid=major,
    legend pos=north west,
]

\addplot[color=red, mark=*, mark options={scale=1.5}] plot coordinates {

(1e-0,  0.476146311265 )
(1e-1,  0.145704900204)
(1e-2,  0.0460515901925 )
(1e-3,   0.014521762778)
(1e-4,  0.00447954695607 )
(1e-5,  0.00128260486954)

};

\addplot[thick, color=black, no markers] plot coordinates {

(1e-2, 1e-1)
(1e-4, 1e-1)
(1e-4, 1e-2)
(1e-2, 1e-1)
};

\end{loglogaxis}
\end{tikzpicture}}
\caption{Error $\|u_\alpha - u^\dagger\|_{L^2(\Omega)}$ for $f_3$ in a double logarithmic plot for different values for $\alpha$. For comparison we plotted a triangle with slope $\frac{1}{2}$.}
\label{fig:ex3}
\end{figure}

\bibliographystyle{plain}
\bibliography{literatur}

\end{document}